\documentclass[conference]{IEEEtran}
\IEEEoverridecommandlockouts
\usepackage{cite}
\usepackage{amsmath,amssymb,amsfonts}
\usepackage{graphicx}
\usepackage{textcomp}
\usepackage{xcolor}
\usepackage{mathtools}
\usepackage{amsthm}
\usepackage{cleveref}
\usepackage{algorithm}
\usepackage{algpseudocode}

\graphicspath{{images/}}

\newcommand{\R}{\mathbb{R}}

\newcommand{\Nc}{\mathcal{N}}

\DeclareMathOperator{\prox}{prox}

\newtheorem{proposition}{Proposition}
\newtheorem{ass}{Assumption}
\newtheorem{corollary}{Corollary}
\newtheorem{remark}{Remark}
\newtheorem{lemma}{Lemma}

\begin{document}

\title{Diffusion at Absolute Zero: Langevin Sampling Using Successive Moreau Envelopes
}

\author{\IEEEauthorblockN{Andreas Habring}
\IEEEauthorblockA{\textit{
Institute of Visual Computing} \\
\textit{Graz University of Technology}\\
\textit{Graz, Austria} \\
\textit{andreas.habring@tugraz.at}}
\and
\IEEEauthorblockN{Alexander Falk}
\IEEEauthorblockA{
\textit{Institute of Visual Computing} \\
\textit{Graz University of Technology}\\
\textit{Graz, Austria} \\
\textit{falk@tugraz.at}}
\and
\IEEEauthorblockN{Thomas Pock}
\IEEEauthorblockA{
\textit{Institute of Visual Computing} \\
\textit{Graz University of Technology}\\
\textit{Graz, Austria }\\
\textit{thomas.pock@tugraz.at}}
}

\maketitle

\begin{abstract}
In this article we propose a novel method for sampling from Gibbs distributions of the form $\pi(x)\propto\exp(-U(x))$ with a potential $U(x)$. In particular, inspired by diffusion models we propose to consider a sequence $(\pi^{t_k})_k$ of approximations of the target density, for which $\pi^{t_k}\approx \pi$ for $k$ small and, on the other hand, $\pi^{t_k}$ exhibits favorable properties for sampling for $k$ large. This sequence is obtained by replacing parts of the potential $U$ by its Moreau envelopes. Sampling is performed in an Annealed Langevin type procedure, that is, sequentially sampling from $\pi^{t_k}$ for decreasing $k$, effectively guiding the samples from a simple starting density to the more complex target. In addition to a theoretical analysis we show experimental results supporting the efficacy of the method in terms of increased convergence speed and applicability to multi-modal densities $\pi$.
\end{abstract}


\section{Introduction}
This article is concerned with Langevin sampling from distributions of the form
\begin{equation}\label{eq:potential}
    \pi(x) = \frac{e^{-U(x)}}{\int e^{-U(y)}\; dy}
\end{equation}
with potentials $U(x) = F(x) + G(x)$ and respective assumptions on $F$ and $G$ (\Cref{ass}). In particular, we are interested in settings where potentially $U$ is not (strongly) convex, non-differentiable and/or not strongly convex such that conventional Langevin sampling techniques are ineffective or too slow. The proposed method is inspired by score based diffusion models \cite{song2020score,song2019generative,ho2020denoising}. More precisely, we propose to successively sample from a deliberately defined sequence of distributions $\pi^t(x)$ for decreasing $t>0$ where $\pi^t$ admits favorable properties for sampling when $t\gg 0$ and converges to the target $\pi$ as $t\rightarrow 0$. Inspired by \cite{durmus2018efficient} this sequence is obtained by approximating the functional $G$ by its Moreau-Yoshida envelope with Moreau-Yoshida parameter $t$.

Sampling from potentials of the form \eqref{eq:potential} is a task frequently arising, for instance, in Bayesian inverse problems, mathematical imaging, machine learning, and uncertainty quantification \cite{zach2023stable,zach2021computed,luo2023bayesian,pereyra2016proximal,durmus2018efficient,narnhofer2022posterior}. In particular, training strategies for machine learning such as maximum likelihood estimation \cite{du2020improved,zach2023stable,nijkamp2020anatomy}, which have become increasingly relevant due to the rise of data driven approaches, often rely on sampling algorithms as subroutines. 
A popular technique for sampling from \eqref{eq:potential} are Markov chain Monte Carlo methods, where a Markov chain $(X_k)_{k\geq1}\subset \R^d$ is carefully designed such that the law of $X_k$ converges to the target distribution $\pi(x)$ as $k\rightarrow\infty$. In the particular case of Langevin sampling such Markov chains are obtained as discretizations of the Langevin diffusion SDE
\begin{equation}\label{eq:langevin_sde}
    dX_t = -\nabla U (X_t) dt + \sqrt{2}dW_t
\end{equation}
where $(W_t)_t$ denotes Brownian motion. While more sophisticated schemes have been proposed a straightforward way to obtain a Markov chain approximating \eqref{eq:langevin_sde} is via the Euler-Maruyama disretization leading to the so-called unadjusted Langevin algorithm (ULA)
\begin{equation}
    X_{k+1} = X_k -\tau\nabla U(X_k) + \sqrt{2\tau} Z_k
\end{equation}
with $\tau>0$ the step-size and $(Z_k)_k$ i.i.d Gaussian random variables. For various convergence results of such schemes we refer the reader to \cite{durmus2019analysis,durmus2019high,durmus2017nonasymptotic,dalalyan2017theoretical,lamberton2003recursive,habring2024subgradient,fruehwirth2024ergodicity}.

\subsection{Denoising Score Matching and Annealed Langevin}
A particular inspiration for the present article has been denoising score matching (DSM) and Annealed Langevin sampling \cite{song2019generative}, which can be classified under the umbrella term \emph{diffusion models} \cite{song2020score}. In particular, in \emph{variance exploding} DSM one considers a sequence of perturbations of the target density $\pi$ defined as convolutions with centered Gaussians with different noise levels $0<t_0<\dots<t_n$, $\pi_{t_k} = \pi*\Nc(0,\sqrt{t_k})$. Thus, samples from $\pi_{t_k}(x)$ can be obtained from samples from $\pi(x)$ by adding a Gaussian random variable. For large $t_k$, the distribution $\pi_{t_k}$ becomes progressively well-behaved, namely smooth, everywhere non-zero, and increasingly \emph{Gaussian}. Therefore, the distributions $\pi_{t_k}(x)$ (respectively, their scores $\nabla\log \pi_{t_k}(x)$) are more stable to train and easier to sample from rather than $\pi$ (for a detailed explanation, see \cite{song2019generative}). In order to sample from $\pi(x)$, these properties can be exploited by sequentially sampling from $\pi_{t_k}$ for $k=n,n-1\dots 0$ using Langevin sampling for each noise level. This technique is coined \emph{Annealed Langevin} sampling. Figuratively speaking, the sampling procedure starts at a well-behaved distribution and is guided towards an increasingly complex one.

\section{Diffusion at Absolute Zero}
For a function $G:\R^d\rightarrow \R$ the Moreau envelope with parameter $t>0$ is defined as 
\begin{equation}\label{eq:moreau}
    M_G^t(x) \coloneqq \inf_{y\in\R^d} G(y)+\frac{1}{2t}\|x-y\|^2.
\end{equation}
Moreover, we define the proximal mapping, or simply \emph{prox}, ${\prox}_{tG}(x)$ of $G$ as the (possibly multi-valued or empty) set of minimizers of \eqref{eq:moreau}. It always holds that $M_G^t\leq G$ and under mild assumptions (which are in particular satisfied under \Cref{ass}) the Moreau envelope is continuous in $(x,t)$ and satisfies $M^t_G(x)\rightarrow G(x)$ as $t\rightarrow 0$ which justifies using it as an approximation of $G$ \cite[Theorem 1.25]{rockafellar2009variational}. Moreover a crucial property of the Moreau envelope is that $\nabla M_G^t(x) \in \frac{1}{t}(x-{\prox}_{t G}(x))$, at every $x$ where $M_G^t$ is differentiable \cite[10.32 Example]{rockafellar2009variational}, which yields a practical way of computing the gradient of $M^t_G$ by solving an optimization problem, at least if the prox is unique.

In many applications (cf. \cite{narnhofer2022posterior}) the potential $U=F+G$ can be split\footnote{Note that it is always possible to define $F\equiv 0$ and $G=U$.} into a rather well-behaved $F(x)$ (e.g. a squared $\ell_2$ norm for inverse problems with Gaussian noise), and a more difficult to handle $G(x)$ (e.g. non-smooth and/or non-convex \cite{kobler2020total}). Therefore, inspired by DSM we propose to consider the sequence of perturbed densities $\pi^t(x) \coloneqq \frac{1}{Z_t} e^{-U^t(x)}$, $t>0$ where
\[
    U^t(x)\coloneqq F(x) + M_G^t(x).
\]
and $Z_t=\int_{\R^d}e^{-U^t(x)}dx$. As a sampling strategy we propose to apply Annealed Langevin sampling to $(\pi^t)_t$ which we coin \emph{diffusion at absolute zero} (DAZ) (see \Cref{sec:relation_to_diffusion}). The method is depicted in \Cref{algo}. There, $0<t_0<t_1\dots<t_N$ denotes a sequence of Moreau parameters and $(\tau_n)_n$ a corresponding sequence of step sizes used in each Langevin sampling loop. A major benefit of the proposed method, which is in particular a reason for increased convergence speed, is that for large $t_n$ larger step sizes can be used without violating the conditions for ergodicity since the gradient of the Moreau envelope $\nabla M_G^t$ is $\frac{1}{t}$-Lipschitz if $G$ is convex (cf. \cite{durmus2018efficient}; for the non-convex case, see the remarks in the Gaussian mixture experiment below). The larger step sizes, in turn, lead to faster mixing of the Markov chain. In contrast to conventional diffusion, DAZ is applicable to any potential $U$ without prior training.
\begin{algorithm}
\footnotesize
\caption{Diffusion at Absolute Zero (DAZ)}\label{algo}
\begin{algorithmic}[1]
\For{$n=N,\dots, 1$}
    \For{$k=0,1,\dots K-1$}
        \State $Z\sim\mathcal{N}(0,I)$
        \State $X_{k+1}^n =$
        \[
            X_k^{n} - \tau_n \nabla F(X_k^n)
                - \frac{\tau_n}{t_n}\left(X_k^n - {\prox}_{t_n G}(X_k^n) \right) + \sqrt{2\tau_n} Z
        \]
    \EndFor
    \State $X_{0}^{n-1} = X_{K}^n$
\EndFor
\end{algorithmic}
\end{algorithm}
While in \cite{durmus2018efficient} the main motivation for using the Moreau envelope for sampling was to handle non-differentiabilites of $G$, in the present work we also want to emphasize the \emph{convexifying} properties of the Moreau envelope. As shown in \Cref{fig:moreau_gm} as $t$ increases, the local maximum of the potential at the center decreases, thus, \emph{connecting} the two separated minima. As we will see in \Cref{sec:experiments} this will enable DAZ to quickly cover all modes of a multi-modal distribution.
\begin{figure}
    \centering
    \includegraphics[width=0.4\textwidth]{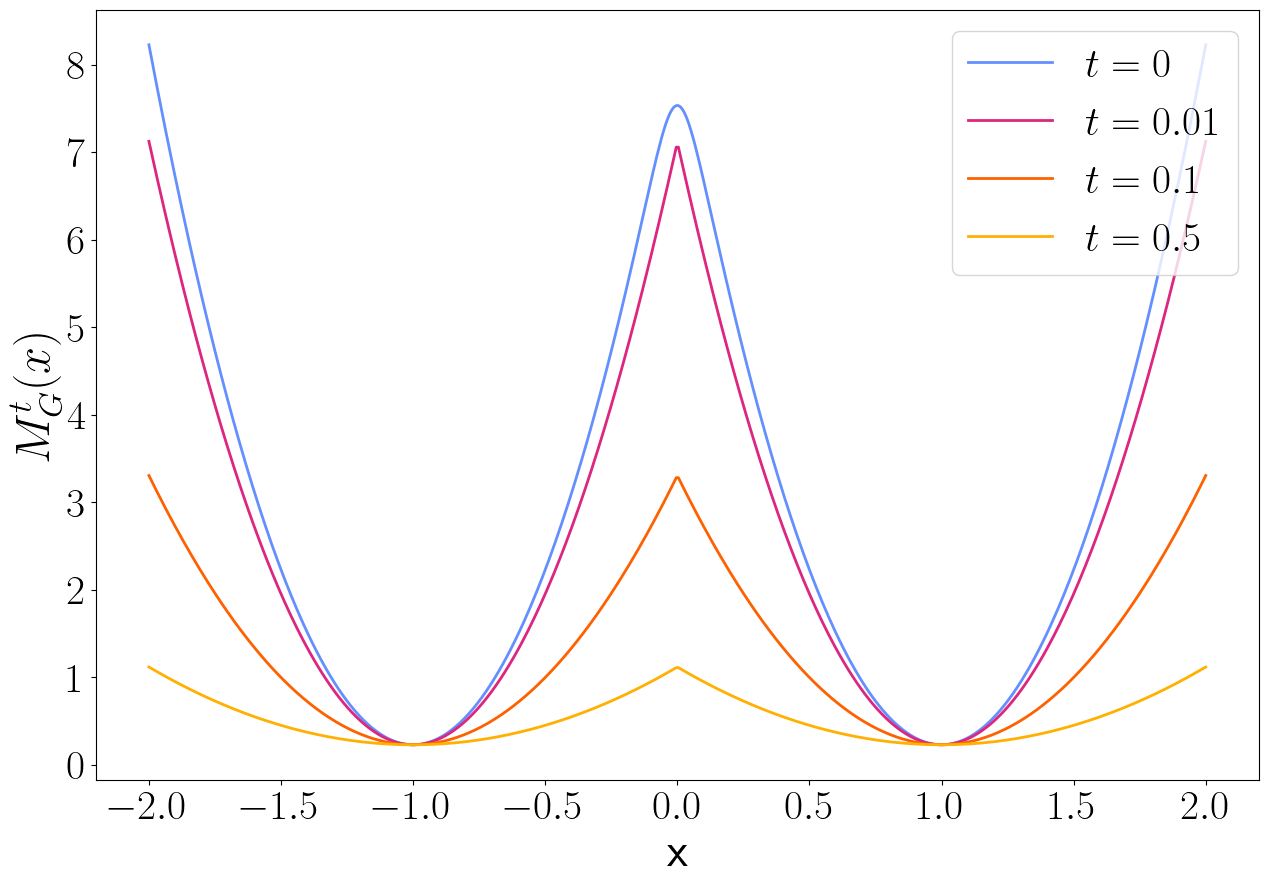}
    \caption{Moreau envelopes of the potential for a Gaussian mixture for different Moreau parameters.}
    \label{fig:moreau_gm}
\end{figure}
\subsection{Consistency of Diffusion at Absolute Zero}
The aim of this section is to show that the proposed sampler DAZ is, in fact, consistent in the sense that the sequence of distributions $(\pi^t)_t$ is well-defined, Lipschitz continuous, and approximates the target $\pi$ as $t\rightarrow0$. In order to analyze the method theoretically, from now on we impose the following assumptions on $F$ and $G$.
\begin{ass}\label{ass}\
    \begin{enumerate}
        \item $F$ and $G$ are bounded from below and such that $\int_{\R^d}\exp(-U(x))dx<\infty$.
        \item $G$ is proper, continuous, coercive, and there exists $t_{max}>0$ with ${\prox}_{tG}(x)$ being single-valued for any $x$ and $t<t_{max}$.
        \item There exists $K>0$ such that for $\|x\|\geq K$, $\|{\prox}_{t G}(x)\|\leq \|x\|$ for any $t\geq 0$.
        \item $G$ admits a regular subgradient $\partial G$ (\cite[Definition 8.3]{rockafellar2009variational}) and the function $\phi(s)\coloneqq \sup \{\|p\|\;|\;\|x\|\leq s, p\in \partial G(x)\}$ satisfies %
        \[\int \phi(\|x\|)^2e^{-U^{t_{max}}(x)}dx<\infty.\]
    \end{enumerate}
\end{ass}

\begin{lemma}\label{lemma:prox_time_cont}
    The mapping $t\mapsto{\prox}_{tG}(x)$ is continuous on $(0,t_{max})$ for any fixed $x$.
\end{lemma}
\begin{proof}
    Let $p_\tau = {\prox}_{\tau G}(x)$ and $t_n\rightarrow t\in(0,t_{max})$. By definition of the proximal mapping and boundedness from below of $G$, $(p_{t_n})_n$ is bounded. Therefore, there exists a (not-relabeled) convergent subsequence $p_{t_n}\rightarrow \hat{p}$. It holds for arbitrary $y\in\R^d$
    \[
        \frac{1}{2t_n}\|x-p_{t_n}\|^2+G(p_{t_n})\leq \frac{1}{2t_n}\|x-y\|^2+G(y).
    \]
    Taking the limit and noting that $G$ is continuous shows that $\hat{p} = p_t$ by uniqueness of the prox (\Cref{ass}). Convergence of the original sequence $(p_{t_n})_n$ to $\hat{p}$ follows by a standard subsequence argument and the fact that ${\prox}_{\tau G}(x)$ is unique.
\end{proof}

\begin{lemma}\label{lemma:moreau_diffble}
    The function $\tau\mapsto M_G^\tau(x)$ is differentiable on $(0,t_{max})$ for any fixed $x\in\R^d$.
\end{lemma}
\begin{proof}
    Denote let $0<s<t<t_{max}$. By definition of the Moreau envelope we find
    \begin{equation}
        \begin{aligned}
            \frac{1}{2t}\|x-p_t\|^2+G(p_t) - \frac{1}{2s}\|x-p_t\|^2-G(p_t)\\
            \leq M_G^t(x)-M_G^s(x)\\
            \leq\frac{1}{2t}\|x-p_s\|^2+G(p_s) - \frac{1}{2s}\|x-p_s\|^2-G(p_s).
        \end{aligned}
    \end{equation}
    It follows 
    \[
        -\frac{1}{2st}\|x-p_t\|^2\leq \frac{M_G^t(x)-M_G^s(x)}{t-s}\leq -\frac{1}{2st}\|x-p_s\|^2.
    \]
    Letting $|s-t|\rightarrow 0$, by \Cref{lemma:prox_time_cont} the result follows and, in particular, we obtain $\partial_t M_G^t(x) = -\frac{1}{2t^2}\|x-p_t\|^2$.
\end{proof}
\begin{remark}
    This shows -- without assuming Lipschitz continuity of $G$ (cf. \cite[Theorem 5, Section 3.3.2]{evans2022partial}) -- that $M_G^t(x)$ satisfies the Hamilton-Jacobi equation
    \[
        \partial_tM_G^t(x) + \frac{1}{2}\|\nabla M_G^t(x)\|^2 = 0.
    \]
\end{remark}
\begin{proposition}\label{prop:Moreau_cont}
    The function $\tau\rightarrow M_G^\tau(x)$ is Lipschitz continuous on $(0,t_{max})$ for any $x\in\R^d$, more precisely, there exists $C>0$ such that for $0<s,t<t_{max}$, $x\in\R^d$
    \begin{equation}
        |M_G^t(x)-M_G^s(x)|\leq \frac{|s-t|}{2}(\phi(\|x\|) + C)^2.
    \end{equation}
\end{proposition}
\begin{proof}
    Using \Cref{lemma:moreau_diffble} we find for $0<s<t$
    \begin{equation}
        \begin{aligned}
            |M_G^t(x)-M_G^s(x)| \leq \int_s^t |\partial_\tau M_G^\tau(x)|\;d\tau \\
            = \frac{1}{2}\int_s^t\frac{1}{\tau^2}\|x-p_\tau\|^2 \;d\tau.
        \end{aligned}
    \end{equation}
    By definition of the Moreau envelope $\frac{1}{\tau}(x-p_\tau)\in\partial G(p_\tau)$. By \Cref{ass} there exists $K>0$, such that for $\|x\|\geq K$, $\|{\prox}_{\tau G}(x)\|\leq \|x\|$ implying $\|\partial G(p_\tau)\|\leq \phi(\|x\|)$. For $\|x\|\leq K$, on the other hand, $\|p_\tau\|$ is bounded by a constant depending on $K$ and thus $\|\partial G(p_\tau)\|\leq C$. Altogether $|M_G^t(x)-M_G^s(x)|  \leq \frac{|t-s|}{2}(\phi(\|x\|) + C)^2$.
\end{proof}
\begin{corollary}\label{cor:cont_normalization}
    The map $\tau\mapsto Z_\tau$ with $Z_\tau=\int_{\R^d}e^{-U^\tau(x)}dx$ is Lipschitz continuous on $(0,t_{max})$.
\end{corollary}

\begin{proposition}\label{prop:curve_lipschitz}
    The curve $(\pi^\tau)_{\tau\geq 0}$ is Lipschitz continuous on $(0,t_{max})$ with respect to the total variation norm.%
\end{proposition}
\begin{proof}
    This follows from \Cref{prop:Moreau_cont}, \Cref{cor:cont_normalization}, and \cite[Theorem 6.15]{villani2009optimal}.

\end{proof}

\subsection{Relation to Denoising Score Matching}\label{sec:relation_to_diffusion}
In order to relate the proposed DAZ to DSM, we define the diffusion potential with temperature parameter $T$, $G^t_T(x)$ as
    \begin{equation}\label{eq:diff_pot}
        G^t_T(x) = -T\log\left(\int_{\R^d} \exp\left(\frac{-G(y) - \frac{1}{2t}\|x-y\|^2}{T}\right)dy\right)
    \end{equation}
which corresponds to a soft-minimum instead of the strict minimum in \eqref{eq:moreau}. The integral in \eqref{eq:diff_pot} matches the potential of a distribution defined via $e^{-G/T}*\Nc(0,tT)$ with $\Nc(0,tT)$ a normal distribution with mean zero and variance $tT$ and $*$ the convolution. For $T=1$ \eqref{eq:diff_pot} is precisely the potential for variance exploding diffusion \cite{song2020score} with target distribution $e^{-G}$. Thus, the proposed method can be understood as the zero temperature limit of diffusion with potential $G^t_T$:
\begin{lemma}
    DAZ is obtained as the zero-temperature limit of regular diffusion. That is, $G^t_T(x)\rightarrow M_G^t(x)$, as $T\rightarrow 0^+$ for any $x$ and, if $\nabla G$ is locally Lipschitz continuous, the convergence is uniformly on compact sets.
\end{lemma}
\begin{proof}
    The result is an extension of the convergence $\|\;\cdot\;\|_p\rightarrow\|\;\cdot\;\|_\infty$ as $p\rightarrow\infty$. 
\end{proof}

\section{Numerical Experiments}\label{sec:experiments}
Since verifying convergence of sampling algorithms is challenging in high dimensions we begin with two small-scale examples for one-dimensional distributions for which error metrics on the space of probability distributions can be estimated with high accuracy. Afterwards we consider two high-dimensional examples, namely denoising with a total variation prior first on a chain and then on images. In all experiments we simulate multiple Markov chains in parallel in order to compute the total variation (TV) distance between the target density and the empirical distribution induced by the samples. The sequence of Moreau parameters $(t_n)_n$ (see \Cref{algo}) is defined using \texttt{NumPy logspace}\footnote{https://numpy.org/doc/2.1/reference/generated/numpy.logspace.html} between the smallest and largest Moreau parameter. Moreover, we only perform one step of ULA per Moreau parameter (i.e., $K=1$ in \Cref{algo}).

\subsection{Experiments in 1D}
For 1D distributions as a ground truth target we use the true density $\pi(x)$ since in this setting the normalization $Z=\int e^{-U(x)}dx$ is known.
\subsubsection{Laplace distribution}
Let the potential be defined as $U(x) = G(x) = \lambda |x|$, $x\in\R$, with $\lambda=1$. In this case, the Moreau envelope is equal to the Huber functional $M_G^t(x) = \lambda|x|-\frac{t\lambda^2}{2}$ for $|x|>t\lambda$ and $M_G^t(x) = \frac{x^2}{2t}$ else. It follows that $\nabla M_G^t$ is Lipschitz continuous with Lipschitz constant $L=1/t$ and in order to sample from $\pi^t$ we can use ULA with step size $\tau=t$. A comparison of the total variation convergence of the proposed DAZ and sampling with MYULA \cite{durmus2018efficient} is depicted in \Cref{fig:mixture_potentials}.

\subsubsection{Gaussian Mixture}
We consider a bimodal (and, thus, not log-concave) Gaussian mixture $\pi(x) = \frac{1}{2}(\Nc(x; \mu_1,\sigma_1) + \Nc(x; \mu_2,\sigma_2))$. We choose $\mu_1 = -1$, $\mu_2=1$ and $\sigma_i = 0.25$, $i=1,2$ (see also \Cref{fig:moreau_gm}). Again, the convergence in TV in comparison to MYULA is depicted in \Cref{fig:mixture_potentials}. In \Cref{fig:mix_hist} we show the sample distribution after 1000 iterations of MYULA and DAZ indicating that DAZ leads to faster mixing due to: a) as shown in \Cref{fig:moreau_gm} with large Moreau parameters the separation between modes is reduced and b) for large Moreau parameters $\nabla U(x)$ grows slower, which allows us to use larger step sizes leading to better mixing (cf. the proof of convergence of the explicit scheme \cite[Proposition 5.3]{fruehwirth2024ergodicity}; note that, whereas in the differentiable case the inverse Lipschitz constant of $\nabla U(x)$ limits the step size, here the upper bound corresponds to the reciprocal \emph{growth constand} of $\nabla U(x)$).

\begin{figure}
    \centering
    \includegraphics[width=0.4\textwidth]{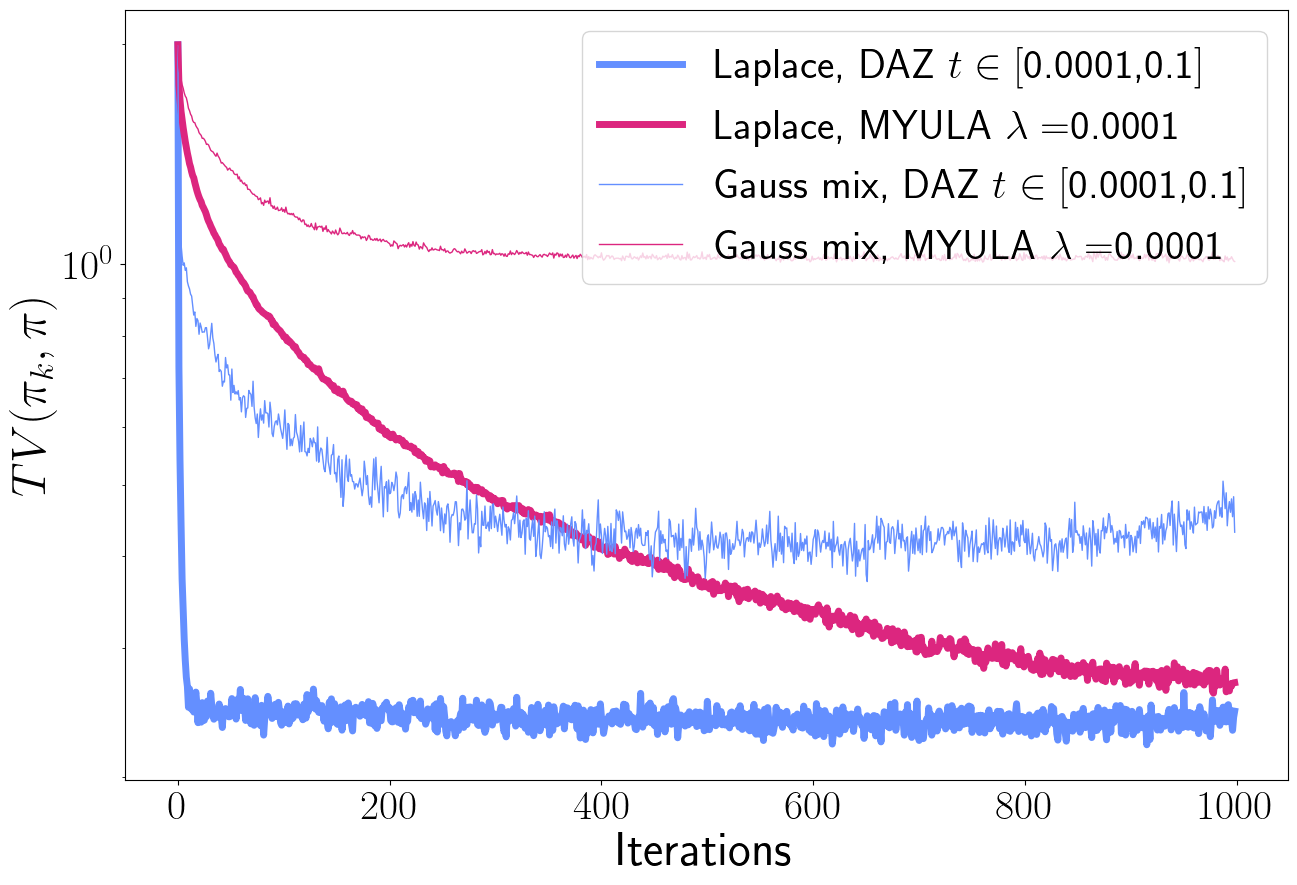}
    \caption{Sampling from a Laplace distribution and a bimodal Gaussian mixture: Comparison of the convergence speed of DAZ and MYULA \cite{durmus2018efficient} with Moreau parameter $\lambda$ with respect to the total variation error. }
    \label{fig:mixture_potentials}
\end{figure}

\begin{figure}
    \centering
    \includegraphics[width=0.4\textwidth]{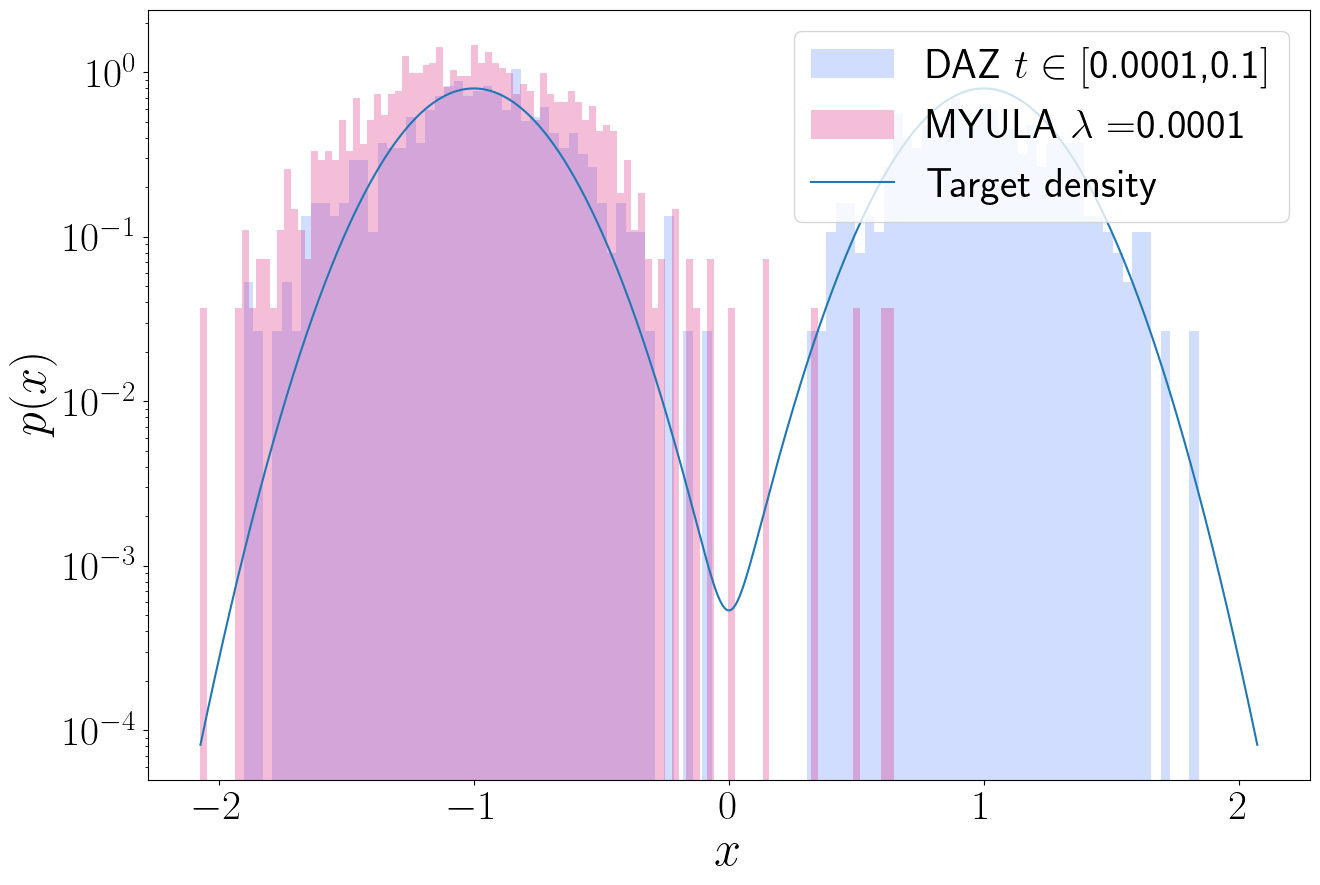}
    \caption{Histogram of the obtained samples for MYULA and DAZ after 1000 iterations compared to the target density.}
    \label{fig:mix_hist}
\end{figure}

\subsection{High-dimensional Experiments}
In higher dimensions evaluation of the sample convergence is rather difficult, as we cannot compute the normalization $Z=\int U(x)dx$ efficiently. Thus, for the following experiments, we use belief propagation (BP) \cite{tapfre03,knobelreiter2020belief,narnhofer2022posterior} to obtain estimates of the marginal distributions with high accuracy to which the obtained samples are compared in TV distance.

\subsubsection{TV-L2 denoising for chains}
We define $F,G:\R^d\rightarrow \R$ as $F(x)\coloneqq \frac{1}{2\sigma^2}\|x-y\|^2$ and $G(x) = \lambda\sum_i |x_{i+1}-x_i|$. We set $d=100$, $\sigma=0.1$, and $\lambda=30$. In order to compute the proximal mapping of $\lambda G$ we use the dynamic programming based direct method proposed in \cite{pock2016total}. The convergence results in TV distance are shown in \Cref{fig:TV_chain}. Note that, $\Pi_i(\pi)$ in the $y$-axis label denotes the $i$-th marginal of the density. Thus, we depict the mean of the TV distances over all marginals. As can be seen, the DAZ again yields a significantly increased convergence speed due to the possibility of using larger step sizes early during the simulation.

\begin{figure}
    \centering
    \includegraphics[width=0.4\textwidth]{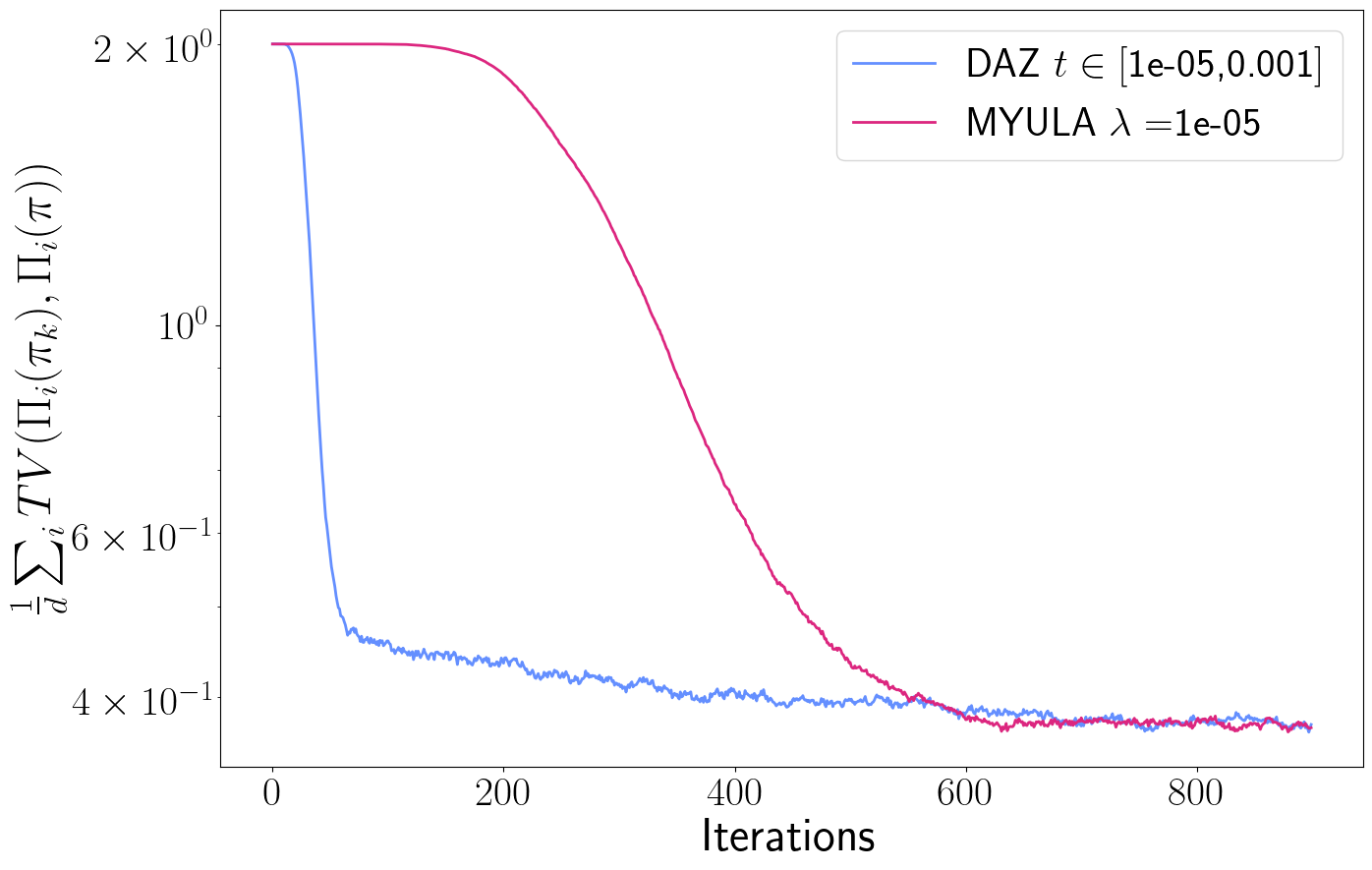}
    \caption{TV denoising on a chain.}
    \label{fig:TV_chain}
\end{figure}

\subsubsection{TV-L2 denoising for images}
For the image denoising example we consider anisotropic total variation for $G$. That is, $F,G:\R^{N\times M}\rightarrow \R$ with $F(x)\coloneqq \frac{1}{2\sigma^2}\|x-y\|^2$ again and $G(x) = \lambda\sum_{i,j} |(Kx)_{i,j}^1|+|(Kx)_{i,j}^2|$ with $K:\R^{N\times M}\rightarrow \R^{N\times M\times 2}$ the finite differences operator \cite[Section 6.1]{chambolle2011first}. We set $\sigma=0.05$, and $\lambda=30$. In order to compute the proximal mapping of $\lambda G$ we use again \cite{pock2016total}. As before we computed the average TV error across all image pixels and the results are depicted in \Cref{fig:TV_image}.
\begin{figure}
    \centering
    \includegraphics[width=0.4\textwidth]{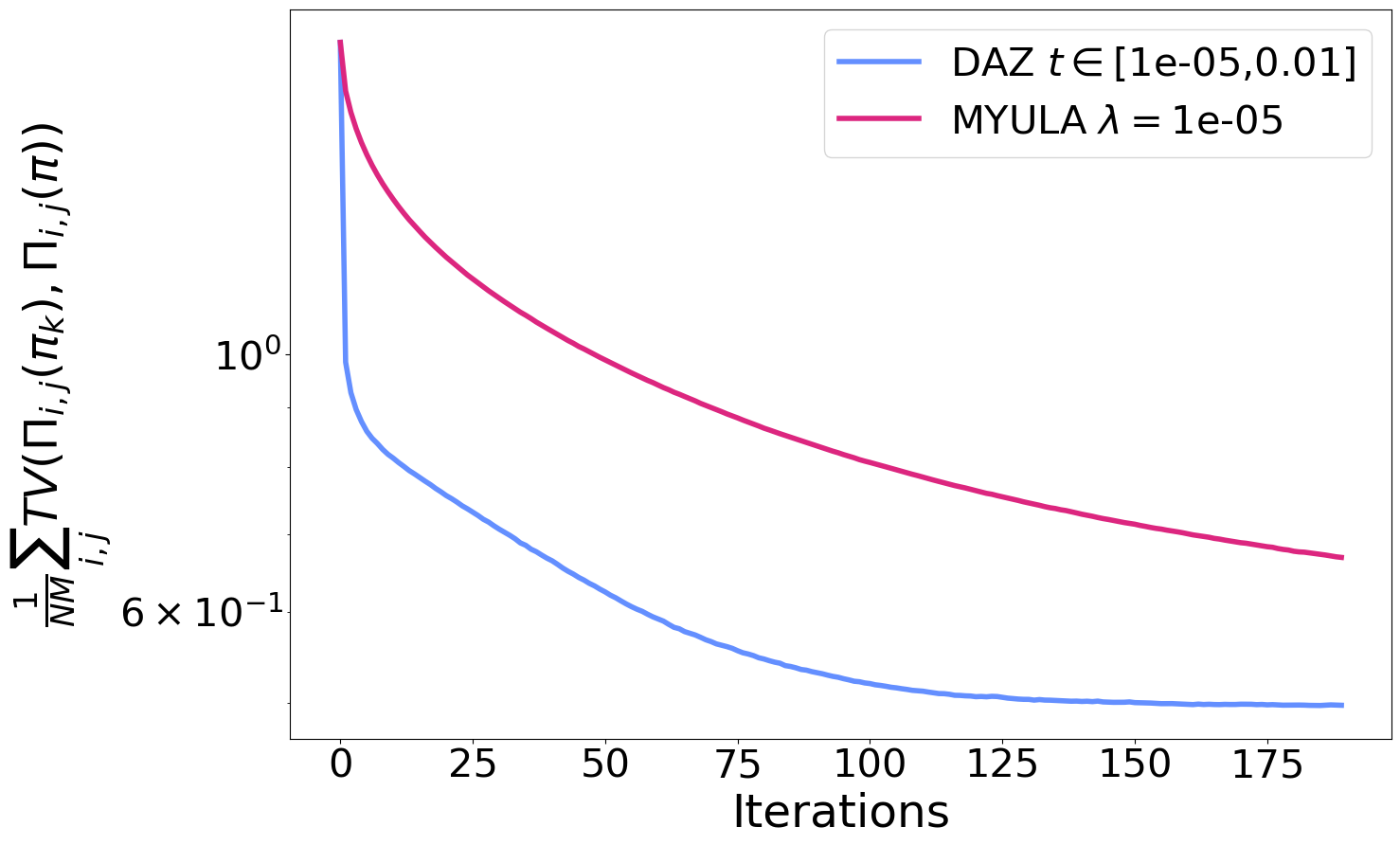}
    \caption{TV denoising on images.}
    \label{fig:TV_image}
\end{figure}

\section{Conclusion and Future Work}
The proposed method shows promising empirical results and we are currently working on a thorough theoretical analysis as well as additional empirical experiments with more complex data driven functionals $G(x)$ such as total deep variation \cite{kobler2020total}.

\bibliographystyle{IEEEtran}
\bibliography{IEEEabrv,references}

\end{document}